\documentclass{amsart}
\usepackage{bm}
\usepackage{amscd, amssymb, amsmath}
\usepackage{amsthm}
\usepackage{graphicx}

\newtheorem{theorem}{Theorem}[section]
\newtheorem{lemma}[theorem]{Lemma}
\newtheorem{proposition}[theorem]{Proposition}
\newtheorem{corollary}[theorem]{Corollary}

\theoremstyle{definition}
\newtheorem{definition}[theorem]{Definition}

\numberwithin{equation}{section}

\newtheorem{remark}[theorem]{Remark}

\newcommand{\Pb}{\mathbb{P}}

\newcommand{\ch}{{\rm ch}}
\newcommand{\Par}{\mathcal{P}}
\newcommand{\Sg}{\mathbb{S}}
\newcommand{\abf}{\mathbf{a}}

\newcommand{\Mmb}{\overline{M}}
\newcommand{\PGL}{{\rm PGL}}

\title{On the cohomology of moduli spaces of 
(weighted) stable rational curves} 
\author{Jonas Bergstr\"om}
\address{Matematiska institutionen, 
Stockholms Universitet, 
106 91 Stockholm, Sweden}
\email{jonasb@math.su.se}
\author{Satoshi Minabe}
\address{Department of Mathematics, 
Tokyo Denki University,
120-8551 Tokyo, Japan}
\email{minabe@mail.dendai.ac.jp}


\begin{document}
\begin{abstract}
We give a recursive algorithm for computing the character of the cohomology of the moduli space $\Mmb_{0,n}$ of stable $n$-pointed genus zero curves as a representation of the symmetric group $\Sg_n$ on $n$ letters. Using the algorithm we can show a formula for the maximum length of this character. Our main tool is connected to the moduli spaces of weighted stable curves introduced by Hassett. 
\end{abstract}
\maketitle
\setlength{\baselineskip}{1.1\baselineskip}

%%%%%%%%%%%%%%%%%%%%%%
\section{Introduction} 
Let $\Mmb_{0,n}$ be the moduli space of stable $n$-pointed curves of genus zero. It parametrizes the isomorphism classes of stable curves $(C, x_1,\ldots, x_n)$, where $C$ is a curve of genus zero and $(x_1, \ldots, x_n)$ are {\it distinct} smooth marked points on it. The symmetric group $\Sg_n$ on $n$ letters acts on $\Mmb_{0,n}$ by permuting the points $x_1, \ldots, x_n$. Hence we have a representation of $\Sg_n$ on the cohomology group $H^*(\Mmb_{0,n}, \mathbb{Q})$. The aim of this paper is to study this representation.

Our first goal is to give a recursive algorithm for computing the character of the $\Sg_n$-module $H^*(\Mmb_{0,n}, \mathbb{Q})$. Formulas for this character have been obtained earlier by Getzler--Kapranov~\cite{GetzlerKapranov} and Getzler~\cite{Getzler} by developing the theory of modular operads. Our method is based on the moduli spaces of weighted stable curves constructed by Hassett~\cite{Hassett}.

A weighted pointed curve consists of a curve and a set of marked points to which rational numbers (called {\it weights}) between zero and one are assigned. Marked points may coincide if the sum of their weights is not greater than one. Hassett constructed moduli spaces for the weighted pointed curves that are \emph{stable}. Since the weights determine the stability conditions for the curves, the moduli spaces may change as the weights are varied. It is shown in \cite{Hassett} that there are induced birational morphisms between these moduli spaces if the weights are varied in certain ways.

Now we explain the idea for computing the character of $H^*(\Mmb_{0,n}, \mathbb{Q})$ using the weighted stable curves. Let $(\mathbb{P}^1)^n/\!\!/ {\rm PGL}(2)$ be the Geometric Invariant Theory (GIT) quotient of the product of $n$ copies of the projective line by the diagonal action of ${\rm PGL}(2)$ with respect to the symmetric linearization $\boxtimes_{i=1}^n \mathcal{O}_{\mathbb{P}^1}(1)$. This space\footnote{More precisely, when $n$ is even, we have to consider the resolution of $(\mathbb{P}^1)^n/\!\!/ {\rm PGL}(2)$ constructed by Kirwan~\cite{Kirwan}, see Lemma \ref{lem:git} below.}, as well as $\Mmb_{0,n}$, can be interpreted as the moduli spaces of weighted stable curves for certain weights. These weights are related in such a way that there is an $\Sg_n$-equivariant birational morphism $\Mmb_{0,n} \to (\mathbb{P}^1)^n/\!\!/{\rm PGL}(2)$. This morphism factorizes into a sequence of blow-ups whose centers are transversal unions of smooth subvarieties.  Furthermore, each component of the centers is isomorphic to a moduli space of weighted stable rational curves with fewer than $n$ marked points. Based on this observation, we give a recursive formula for computing
the $\Sg_n$-character of $H^*(\Mmb_{0,n}, \mathbb{Q})$, see Theorem \ref{thm:ind} for further details.

The second goal of this paper is to show a formula for the length of the $\Sg_n$-module $H^*(\Mmb_{0,n}, \mathbb{Q})$. Recall that the irreducible $\Sg_n$-modules are indexed by the partitions of $n$. By definition, the length of an irreducible $\Sg_n$-module is the number of parts in the corresponding partition. For a finite-dimensional $\Sg_n$-module $V$, we denote the maximum of the lengths of the irreducible representations occurring in $V$ by $l(V)$. In \cite{Faber-Pandharipande}, Faber and Pandharipande showed that 
\begin{equation*}
l\bigl(H^{2i}(\Mmb_{0,n}, \mathbb{Q})\bigr) \leq {\rm min} (i+1,n-i-2). 
\end{equation*}
Using our algorithm we show that the above bound is indeed sharp. Actually, we can show a refined statement, see Theorem \ref{thm:length}.

The paper is organized as follows. In Section~\ref{sec:hassett} we summarize necessary facts on weighted stable curves. The recursive formula for the character of $H^*(\Mmb_{0,n}, \mathbb{Q})$ is presented in Section~\ref{sec:blowup}. It is made more explicit in Section~\ref{sec:algorithm} using symmetric functions, where we also give some examples. We then prove the formula for $l\bigl(H^{2i}(\Mmb_{0,n},\mathbb{Q})\bigr)$ in Section~\ref{sec:length}. 
In Appendix \ref{sec:cohomology} we give a formula for the cohomology of the blow-up
used to derive the formula for $H^*(\Mmb_{0,n}, \mathbb{Q})$ in Theorem \ref{thm:ind}. 
In Appendix~\ref{sec:character} we recall some facts about symmetric functions and characters of symmetric groups. In Appendix \ref{sec:order} we prove a formula for the length of some induced representation used in the proof of Theorem \ref{thm:length}. 

\subsection*{Acknowledgement} 
The second named author is supported in part by JSPS Grant-in-Aid for Young Scientists (No.\,22840041). The authors thank the Max--Planck--Institut f\"ur Mathematik for hospitality. The authors also thank the referees for their thorough reading and helpful comments.

%%%%%%%%%%%%%%%%%%%%%%%
\section{Preliminaries} \label{sec:hassett}
We recall the definition of weighted pointed stable curves and some properties of their moduli spaces. In the following, a curve is a compact and connected curve over $\mathbb{C}$ with at most nodal singularities and the genus of a curve is the arithmetic one.

\subsection{Weighted pointed stable rational curves} \label{sec:moduli}
\begin{definition}\label{def:definition1}
Let $n\geq 3$. An element $\abf=(a_1, \ldots, a_n)\in \mathbb{Q}^n$ 
satisfying   
$0<a_j\leq 1$ and $a_1+\cdots +a_n>2$
is called a {\it weight}.
\end{definition} 
\begin{definition}\label{def:definition2}
Given a weight $\abf=(a_1, \ldots, a_n)$, 
an $\abf$-weighted stable $n$-pointed rational curve 
$(C, x_1,\ldots, x_n)$ consists of a curve $C$ of genus $0$ 
and $n$ marked points $(x_1, \ldots, x_{n})$ 
on $C$ which satisfy the following conditions:
\begin{itemize}
\item[(i) ] all the marked points are non-singular points of $C$,
\item[(ii)] $x_{i_1}=\cdots=x_{i_k}$ may happen if
$\sum_{j=1}^{k}a_{i_j}\leq 1$,  where $\{i_1,\ldots, i_k\}\subset \{1, \ldots, n\}$,
\item[(iii)](\emph{stability}) the number of singularities plus the sum of the weights of the marked points on each irreducible component of $C$ should be strictly greater than $2$. 
\end{itemize} 
\end{definition}

Let $\Mmb_{0, \abf}$ be the moduli space of $\abf$-weighted stable $n$-pointed rational curves. It exists and it is an irreducible smooth projective variety of dimension $n-3$, see~\cite[Theorem~2.1]{Hassett}. For $\abf=(1,\ldots, 1)$ we have $\Mmb_{0, \abf}=\Mmb_{0, n}$. 

\subsection{Reduction morphisms}
Let $\abf=(a_1,\ldots, a_n)$ and 
${\bf b}=(b_1,\ldots,b_n)$ be two weights of the same length
with $a_i\geq b_i$ for all $i$. There exists a 
birational surjective map
$$\rho_{\bf b, a}: \Mmb_{0, \abf} \to \Mmb_{0, {\bf b}}$$
called a reduction morphism, see~\cite[Theorem 4.1]{Hassett}. 
For $(C,x_1\ldots, x_n)\in \Mmb_{0, \abf}$ we obtain 
$\rho_{\bf b, a} (C,x_1\ldots, x_n)\in \Mmb_{0, {\bf b}}$ 
by successively collapsing the components of $C$ which become 
unstable with respect to the weight ${\bf b}$. 
That is, the components of $C$ are collapsed 
for which condition (iii) of Definition~\ref{def:definition2} 
is no longer fulfilled when the weight $\abf$ is replaced by the 
weight ${\bf b}$. 

\begin{remark} 
Throughout this paper, the coefficients of all cohomology groups will be $\mathbb{Q}$. 
Note that the pullback
$\rho_{\bf b, a}^*: H^*(\Mmb_{0, {\bf b}}) \to H^*(\Mmb_{0, \abf})$
is injective 
(see e.g. \cite[Lemma 7.28]{Voisin})  
and that $\Mmb_{0, n}$ has no odd cohomology by a result 
of Keel \cite[p.~549]{Keel}. 
It follows that $\Mmb_{0, \abf}$ has no odd cohomology 
for any weight $\abf$, see also \cite[Theorem 1]{Ceyhan}.
\end{remark}

%%%%%%%%%%%%%%%%%%%%%%%%%%%%%
\section{The blow-up formula} \label{sec:blowup}
\subsection{The blow-up sequence}
For a weight 
$$\abf=(
\underbrace{1,\ldots, 1}_{k},
\underbrace{\dfrac{1}{l},\ldots, \dfrac{1}{l}}_{n-k}),\quad 0\leq k \leq n, 
$$
let $\Mmb^n_{k,l}:=\Mmb_{0, \abf}$ be the moduli space 
of $\abf$-weighted stable rational curves and note that $\Mmb^n_{0,1} \cong \Mmb_{0, n}$. The subgroup $\Sg_k \times \Sg_{n-k}$
of $\Sg_n$ preserves 
the weight $\abf$ and therefore acts on $\Mmb^n_{k,l}$. 
We have the following sequence
of reduction morphisms:
\begin{equation}\label{eq:blow-up}
\Mmb^n_{k,1} \to \Mmb^n_{k,2}\to \cdots \to \Mmb^n_{k, r(n,k)}, 
\end{equation}
where 
$$ 
r(n,k):=\begin{cases}
\lfloor \frac{n-1}{2} \rfloor~~ \mathrm{if}~~ k=0,\\ 
n-2~~ \mathrm{if}~~ k=1,\\
n-k  ~~\mathrm{if}~~ k\geq2.
\end{cases}
$$
The sequence \eqref{eq:blow-up} is equivariant under the action of 
$\Sg_k \times \Sg_{n-k}$. 
If $k\geq 2$ then clearly $\Mmb^n_{k, l} \to \Mmb^n_{k,l+1}$  
is an isomorphism for $l \geq r(n,k)$. 
Moreover, the first reduction morphism  
$\Mmb^n_{k, 1} \to \Mmb^n_{k, 2}$
is an isomorphism, because the fibers consist of products of the space $\Mmb_{0,3}$, which is isomorphic to a point. 

We describe the exceptional locus 
of an arrow $\rho^n_{k,l}: \Mmb^n_{k,l}\to \Mmb^n_{k,l+1}$ in the sequence \eqref{eq:blow-up}.
For a subset $I=\{i_1, \ldots, i_{l+1}\}$ of $\{k+1, \ldots, n\}$, let 
$\Mmb^n_{k,l+1}(I)$ be the closure in $\Mmb^n_{k,l+1}$ of the locus where
the marked points $x_{i_1}, \ldots, x_{i_{l+1}}$ collide.

\begin{lemma}\label{lem:center}
The reduction morphism 
$\rho^n_{k,l}: \Mmb^n_{k,l}\to \Mmb^n_{k,l+1}$
is the blow-up along the union $\cup_I \Mmb^n_{k,l+1}(I)$
where $I$ runs over all subsets of $\{k+1,\ldots, n\}$
of cardinality $l+1$.
\end{lemma}
\begin{proof}
See \cite[Proposition 4.5 and Remark 4.6]{Hassett}. 
\end{proof}

\subsection{Intersections of components} 
Now we describe intersections between the spaces 
$\Mmb^n_{k,l+1}(I)$ for different choices of $I$. 
Let $I_1, \ldots, I_m$ be any subsets of 
$\{k+1,\ldots, n\}$ such that $\# I_j=l+1$, then 
$\cap_{j=1}^m \Mmb^n_{k,l+1}(I_j) \neq \emptyset$
if and only if $I_j \cap I_{j'}=\emptyset$ for all $j\neq j'$. 
When this condition is satisfied we have
$$
\cap_{j=1}^m \Mmb^n_{k,l+1}(I_j) \cong \Mmb^{n-lm}_{k+m, l+1}~,
$$
which follows from \cite[Proposition 4.5]{Hassett}. Let $G^{n}_{k, m, l+1}$ be the subgroup of $\Sg_k \times \Sg_{n-k}$ which preserves
the set $\{I_1, \ldots, I_m\}$. We have
$$
G^{n}_{k, m, l+1} \cong 
\Sg_k \times 
[(\Sg_{l+1})^m \rtimes \Sg_m] \times \Sg_{n-k-(l+1)m}~.
$$
The group $G^{n}_{k, m,  l+1}$ acts on the intersection $\cap_{j=1}^m \Mmb^n_{k,l+1}(I_j)$. 

\subsection{The main theorem}
As we have mentioned, $\Sg_{k}\times \Sg_{n-k}$ acts on $\Mmb^n_{k, l}$
by permuting the marked points. 
This makes the cohomology groups $H^i(\Mmb^n_{k, l})$
into representations of $\Sg_k \times \Sg_{n-k}$. 

Let us consider the graded representation
$$
{\rm Res}^{\Sg_{k+m}\times \Sg_{n-k-m(l+1)}}_
{\Sg_k \times \Sg_m \times \Sg_{n-k-m(l+1)}}
H^*(\Mmb^{n-lm}_{k+m, l+1})~.
$$
We regard this as a graded representation of $G^{n}_{k, m, l+1}$, 
where $(\Sg_{l+1})^m$ acts trivially, and we denote it 
by $W^{n}_{k, m, l+1}$.

\begin{theorem}\label{thm:ind} 
We have the following identity 
in the ring of graded  
representations of $\Sg_k\times \Sg_{n-k}$: 
\begin{equation} \label{eq:ind}
H^*(\Mmb^n_{k,l})=H^*(\Mmb^n_{k,l+1})
\oplus 
\bigoplus_{m=1}^{\lfloor \frac{n-k}{l+1} \rfloor}
{\rm Ind}^{\Sg_k\times \Sg_{n-k}}_{G^{n}_{k, m, l+1}}~~ 
\left( W^{n}_{k, m, l+1} \otimes \right(H^+(\mathbb{P}^{l-1})\left)^{\otimes m}
\right)~,  
\end{equation}
where on the right hand side, $H^+$ is the part of the cohomology with positive degree. 
Moreover we regard $(H^+(\mathbb{P}^{l-1}))^{\otimes m}$, 
on which $\Sg_m$ acts by permutation of the factors, 
as a representation of $G^{n}_{k, m, l+1}$ 
where every factor except $\Sg_m$ acts trivially.  
\end{theorem}
\begin{proof}
By Lemma \ref{lem:center}, the reduction morphism
$\rho^n_{k,l}: \Mmb^n_{k,l}\to \Mmb^n_{k,l+1}$
is the blow-up along the transversal union of smooth subvarieties of 
codimension $l$. Therefore we can apply the result in 
Appendix \ref{sec:cohomology}.  The theorem 
follows from Proposition \ref{prop:cohomology}
after taking into account the action of $\Sg_k\times \Sg_{n-k}$.
\end{proof}

Formula~\eqref{eq:ind} can be used inductively, knowing the $\Sg_m$-representation $\bigl(H^+(\mathbb{P}^{l-1})\bigr)^{\otimes m}$ (see Lemma~\ref{lem:PF}) and using the fact that $H^*(\Mmb^n_{k,1})$ equals ${\rm Res}^{\Sg_n}_{\Sg_k \times \Sg_{n-k}} H^*(\Mmb^n_{0,1})$.

\begin{corollary} \label{cor:ind} 
By induction, with the base cases being $\Mmb^n_{0, r(n,0)}$ for all $n\geq 3$,
we can use formula \eqref{eq:ind} to compute the graded $\Sg_k\times \Sg_{n-k}$-representation $H^*(\Mmb^n_{k,l})$ for any $n$, $k$
and~$l$.
\end{corollary}

%%%%%%%%%%%%%%%%%%%%%%%
\section{The algorithm}\label{sec:algorithm}
\subsection{Poincar\'{e}-Serre polynomials} 
In this section we will make Corollary~\ref{cor:ind} more explicit using 
Poincar\'{e}-Serre polynomials, see Appendix~\ref{sec:character} for the notation. 

\begin{definition}
The $(\Sg_k \times \Sg_{n-k})$-equivariant Poincar\'{e}-Serre polynomial of 
$\Mmb^n_{k, l}$ is defined by 
\begin{equation*}
E^n_{k, l}(q):=\sum_{i=0}^{n-1}
\ch_{k, n-k}\left(H^{2i}(\Mmb^n_{k,l})\right) q^{i}\in \Lambda^{x, y} [q]~.
\end{equation*}
\end{definition}

It is straightforward to show the following lemma. 
\begin{lemma}\label{lem:PF}
The graded $\Sg_m$-character of $\bigl(H^+(\mathbb{P}^{l-1})\bigr)^{\otimes m}$
is given in $\Lambda^y[q]$ by
\begin{equation*}
F^y_{m,l}(q):=\sum_{k=1}^{m(l-1)}
\left(
\sum_{(m_1, \cdots, m_{l-1}) \in I_{l,k}}
\; \prod_{j=1}^{l-1} s^y_{(m_j)}
\right)
q^k
~,
\end{equation*}
where $I_{l,k}:=\{(m_1, \cdots, m_{l-1}): m_i \geq 0,\sum_{i} i m_i=k, \sum_{i} m_i=m \}$.
\end{lemma}

\begin{proposition} \label{prop:algor}
For any $\nu \in \Par(m)$ we write $\frac{\partial}{\partial{p^x_{\nu}}} E^{n-lm}_{k+m,l+1}(q)=\sum_{\lambda,\mu} a_{n,k,m,l}^{\nu,\lambda,\mu}(q) \, p^x_{\lambda} p^y_{\mu} $ for some $a_{n,k,m,l}^{\nu,\lambda,\mu}(q) \in \mathbb{Q}[q]$. 
We then have 
  \begin{equation}
    \label{eq:algor}
E^n_{k,l}(q)=E^n_{k,l+1}(q)+
\sum_{m=1}^{\lfloor \frac{n-k}{l+1} \rfloor}
\sum_{\nu \in \Par(m)} \sum_{\lambda,\mu} a_{n,k,m,l}^{\nu,\lambda,\mu}(q) \, p^x_{\lambda}p^y_{\mu} \, \Bigl(\bigl(p^y_{\nu} * F^y_{m,l}(q) \bigr) \circ s^y_{(l+1)}\Bigr)  ~,        
  \end{equation}
where $*$ and $\circ$ act trivially on $q$.
\end{proposition}
\begin{proof} This follows directly from Theorem~\ref{thm:ind} using 
Appendix~\ref{sec:character}.
\end{proof}

\subsection{Base cases}
Here we will find formulas for $E^n_{0, r(n, 0)}(q)$ which are the base cases in the induction. 

\begin{lemma}[\cite{KiemMoon}, Theorem 1.1] \label{lem:git}
 $\,$ \\
(i) If $n$ is odd,  
$\Mmb^n_{0, r(n, 0)}$
is isomorphic to the GIT quotient 
$(\mathbb{P}^1)^n/\!\!/\PGL(2)$.
\\
(ii) If $n$ is even, $\Mmb^n_{0, r(n, 0)}$
is isomorphic to Kirwan's desingularization of  
$(\mathbb{P}^1)^n/\!\!/\PGL(2)$ constructed in~\cite{Kirwan}.
\end{lemma}

\begin{remark} \label{rmk:proj}
If $n \leq 6$ we see that $\Mmb^n_{0, r(n, 0)}$ is isomorphic to $\Mmb^n_{0, 1}$. Note also that $\Mmb^n_{1, r(n,1)}$ is isomorphic to $\Pb^{n-3}$ and that 
$$E^n_{1, r(n, 1)}(q)=s_{(1)}^x s_{(n-1)}^y \,\sum_{i=0}^{n-3}q^i~.$$
\end{remark} 

\begin{definition}
For $n\geq 1$ we define
$$P_n(q):=\sum_{i=0}^{\lfloor \frac{n}{2}\rfloor} s_{(n-i)}  s_{(i)} \ (q^{n-i}-q^{i+1})
\in \Lambda[q]~.$$
\end{definition}

\begin{lemma}\label{lem:git2} 
If we suppose $n=2m+1$, then 
$$
E^n_{0,m}(q)=\frac{P_n(q)}{q^3-q}.
$$
\end{lemma}
\begin{proof} 
This formula is an $\Sg_n$-equivariant version of  
in~\cite[Proposition 6.1]{Thaddeus}. 
We first observe that the Poincar\'e-Serre polynomial of 
$(\Pb^1)^n$ equals $s_{(n)} \circ \bigl((q+1)s_{(1)}\bigr)=\sum_{i=0}^n s_{(n-i)}s_{(i)}q^i$. 
The unstable points in 
$(\Pb^1)^n$ under the action of $\PGL(2)$ are the ones for which at 
least $m+1$ coordinates are equal. The loci of points in $(\Pb^1)^n$ 
where at least $m+1$ coordinates are equal to a fixed point $x \in \Pb^1$ 
are disjoint for each choice of $x$. Thus, we find the Poincar\'e-Serre polynomial 
of the locus of unstable points to be $(q+1)\sum_{i=0}^m s_{(n-i)}s_{(i)}q^i$. 
Removing this loci and dividing by $\PGL(2)$  gives the answer, using the additivity of Poincar\'e-Serre polynomials (defined using the Euler-characteristic in the non-projective case). \end{proof}

\begin{lemma}\label{lem:git3} 
Suppose $n=2m$, then we have
$$
E^n_{0, m-1}(q)= 
\frac{P_n(q)
-s_{(m)}^2q^{m}
+(s_{(2)} \circ s_{(m)}) q
+(s_{(1,1)}\circ s_{(m)})q^2}{q^3-q}
+s_{(2)}\circ \Bigl(s_{(m)}\sum_{i=0}^{m-2}q^i \Bigr)~,
$$
where $p_i \circ (p_j \, q^k):=p_{ij}\, q^{ik}$.
\end{lemma}
\begin{proof}
This formula is an $\Sg_n$-equivariant version of the formula in \cite[Section~9.1]{Kirwan}. As in the odd case we begin by considering the locus in $(\Pb^1)^n$ of unstable points, that is, points for which at least $m+1$ of the coordinates are equal. The Poincar\'e-Serre polynomial of this locus equals $(q+1)\sum_{i=0}^{m-1} s_{(n-i)}s_{(i)}q^i$. The strictly semi-stable 
points of $(\Pb^1)^n$ are the ones for which one can find a point in $\Pb^1$ such that precisely $m$ of the coordinates are equal to this point. Consider first the sub-locus of strictly 
semi-stable points in $(\Pb^1)^n$ where there are two distinct points such that there are $m$ of the coordinates equal to each one of them. Its contribution equals $(s_{(2)} \circ s_{(m)})q^2 +(s_{(1,1)} \circ s_{(m)})q$, which follows from the fact that the Poincar\'e-Serre polynomial of $(\Pb^1)^2$ is $s_{(2)} (q^2+q+1)+s_{(1,1)}q$. The contribution from the rest of the semi-stable locus is directly found to be $s_{(m)}^2 (q+1)(q^m-q)$. Thus the contribution from the stable locus is given by
\begin{multline*}
\sum_{i=0}^{2m} s_{(2m-i)} s_{(i)} q^i \, -\, (q+1)\sum_{i=0}^{m-1} s_{(2m-i)} s_{(i)} q^i
\, -\, s_{(m)}^2(q+1)(q^m-q)
\,-\, \left( s_{(2)}\circ s_{(m)}\right) q^2\, \\ -\, \left( s_{(1,1)}\circ s_{(m)}\right) q
=P_n(q) - s_{(m)}^2q^m+\left( s_{(2)}\circ s_{(m)}\right) q+\left( s_{(1,1)}\circ s_{(m)}\right) q^2.
\end{multline*}
Dividing this by $\PGL(2)$, we obtain the first term of the formula.
Note that we used that $s_{(m)}^2=s_{(2)}\circ s_{(m)}+s_{(1,1)}\circ s_{(m)}$. 

We should then add the contribution coming from the moduli space $M$ of weighted curves with two components, each component having $m$-marked points of weight $1/(m-1)$. The node on such a component can be viewed as a point of weight $1$ and hence $M$ consists of copies of $\Mmb^{m+1}_{1,m-1} \times \Mmb^{m+1}_{1,m-1}$, one for each choice of distribution of labels of marked points on each component. The $\Sg_n$-representation $H^*(M)$ will then equal the induced representation 
${\rm Ind}^{\Sg_n}_{\Sg_2 \sim \Sg_m} H^*(\Mmb^{m+1}_{1,m-1})^{\otimes 2}$, where $\Sg_2$ acts trivially. 
Since $E^{m+1}_{1, m-1}(q)=s_{(1)}^x s_{(m)}^y \,(\sum_{i=0}^{m-2}q^i)$, by Remark~\ref{rmk:proj}, we find the Poincar\'e-Serre polynomial of $M$ to be $s^y_{(2)}\circ \bigl(\frac{\partial}{\partial{p^x_{1}}} E^{m+1}_{1, m-1}(q)\bigr)$.
\end{proof}

\subsection{Examples}
The first case for which $\Mmb^n_{0,1}$ is not isomorphic to $\Mmb^n_{0,r(n,0)}$ is when $n=7$. We will now use equation~\eqref{eq:algor} to compute the cohomology of this space. 

To apply equation~\eqref{eq:algor} for $n=7$ and $l=2$ we first need to find $E^3_{2,3}(q)$ and $E^5_{1,3}(q)$. We clearly have $E^3_{2,3}(q)=s^x_{(2)} s^y_{(1)}$. Lemma~\ref{lem:git2} tells us that $E^5_{0,1}(q)=E^5_{0,2}(q)=(q^2+q+1)s^y_{(5)}+qs^y_{(4,1)}$, and since $F^y_{1,2}=qs^y_{(1)}$ and $H^*(\Mmb^5_{1,2})={\rm Res}^{\Sg_5}_{\Sg_{1} \times \Sg_{4}} H^*(\Mmb^5_{0,1})$, applying equation~\eqref{eq:algor} for $n=5$ and $l=2$ gives 
$$E^5_{1,3}(q)=s^x_{(1)} \, \Bigl(\frac{\partial}{\partial p^y_1} E^5_{0,2}(q)\Bigr)-s^x_{(1)}s^y_{(1)}(qs^y_{(1)} \circ s^y_{(3)}) = s^x_{(1)}s^y_{(4)}(q^2+q+1)~.$$
Again, Lemma~\ref{lem:git2} tells us that 
$$E^7_{0,3}(q)=(q^4+q^3+2 q^2+q+1)s_{(7)}+(q^3+q^2+q)s_{(6,1)}+q^2s_{(5,2)}~,$$
and since $F^y_{2,2}=q^2s^y_{(2)}$, applying equation~\eqref{eq:algor} for $n=7$ and $l=2$ gives
\begin{multline*}
  E^7_{0,1}(q)=E^7_{0,2}(q)=E^7_{0,3}(q)+s_{(4)}(q^2+q+1)\, qs_{(3)} 
+ s_{(1)} \, q^2(s_{(2)} \circ s_{(3)})
\\
=s_{(7)}(q^4+2q^3+4q^2+2q+1)+s_{(6,1)}(2q^3+3q^2+2q)+s_{(5,2)}(q^3+3q^2+q)+\\+s_{(4,3)}(q^3+2q^2+q)+s_{(4,2,1)}q^2~.
\end{multline*}

We now turn to $n=8$, which is the first interesting even case. To apply equation~\eqref{eq:algor} for $n=8$ and $l=2$ we need to find $E^4_{2,3}(q)$ and $E^6_{1,3}(q)$. The map $\Mmb^4_{2,2} \rightarrow \Mmb^4_{2,3}$ is an isomorphism, $H^*(\Mmb^4_{2,2})={\rm Res}^{\Sg_4}_{\Sg_{2} \times \Sg_{2}} H^*(\Mmb^4_{0,1})$ and so $E^4_{2,3}(q)=(q+1)s^x_{(2)}s^y_{(2)}$. Then it follows from Lemma~\ref{lem:git3} that 
$$E^6_{0,1}(q)=E^6_{0,2}(q)=(q^3+2q^2+2q+1) s^y_{(6)}+(q^2+q)s^y_{(5,1)}+(q^2+q)s^y_{(4,2)}$$
and applying equation~\eqref{eq:algor} for $n=6$ and $l=2$ we get 
\begin{multline*}
E^6_{1,3}(q)=s^x_{(1)} \, \Bigl(\frac{\partial}{\partial p^y_1} E^6_{0,2}(q)\Bigr)-(q+1)s^x_{(1)}s^y_{(2)}\, qs^y_{(3)}=\\=s^x_{(1)}\bigl( (q^3+2q^2+2q+1) s^y_{(5)}+(q^2+q)s^y_{(4,1)} \bigr).
\end{multline*}
Again we use Lemma~\ref{lem:git3} to see that 
\begin{multline*}
E^8_{0,3}(q)=(q^5+2q^4+3q^3+3q^2+2q+1)s_{(8)}+(q^4+2q^3+2q^2+q)s_{(7,1)}+(q^4+2q^3+2q^2+q)s_{(6,2)}+\\+(q^3+q^2)s_{(5,3)}+(q^4+q^3+q^2+q)s_{(4,4)}.
\end{multline*}
We then apply equation~\eqref{eq:algor} for $n=8$ and $l=2$, which gives 
\begin{multline*}
E^8_{0,1}(q)=E^8_{0,2}(q)=E^8_{0,3}(q)+ \Bigl(\frac{\partial}{\partial p^x_1} E^6_{1,3}(q)\Bigr) qs^y_{(3)}+(q+1)s^y_{(2)} \,q^2(s^y_{(2)} \circ s^y_{(3)})=\\
=(q^5+3q^4+6q^3+6q^2+3q+1)s_{(8)}+(2q^4+6q^3+6q^2+2q)s_{(7,1)}+(2q^4+7q^3+7q^2+2q)s_{(6,2)}+\\
+(q^3+q^2)s_{(6,1,1)}+(q^4+5q^3+5q^2+q)s_{(5,3)}+(2q^3+2q^2)s_{(5,2,1)}+\\
+(q^4+3q^3+3q^2+q)s_{(4,4)}+(2q^3+2q^2)s_{(4,3,1)}+(q^3+q^2)s_{(4,2,2)}.
\end{multline*}

%%%%%%%%%%%%%%%%%
\section{Lengths} 
\label{sec:length}
The length of an irreducible representation of $\Sg_n$ is defined to be the length of the corresponding partition, this is written $l(V_{\lambda})=l(\lambda)$ in the notation of Appendix~\ref{sec:character}. We then define the length $l(V)$ of a representation $V$ of $\Sg_n$ to be the maximum length of all of its irreducible constituents. 

It is shown in~\cite[Theorem 4.2]{Faber-Pandharipande} that
\begin{equation}
  \label{eq:maxlength}
  l\bigl(H^{2i}(\Mmb_{0,n})\bigr) \leq {\rm min} (i+1,n-i-2), 
\end{equation}
and this fact is used to prove similar bounds for higher genera. These bounds are in turn used to detect non-tautological classes in the cohomology of $\Mmb_{g, n}$ for higher genera. Here we use Theorem~\ref{thm:ind} to show a result that implies that this bound is indeed sharp, compare the discussion in \cite[Section 4.4]{Faber-Pandharipande}. 

We will say that a partition $\lambda$ has property $\star_{n,i}$, if $\lambda'_1=\lambda_2'={\rm min} (i+1,n-i-2)$ and $\lambda'_3\geq 1$. 
In the following, we use the notation of Appendix~\ref{sec:order}. 
\begin{theorem}\label{thm:length} 
Let us put $w_{n,i}:=w\bigl(\mathrm{ch}_n\left(H^{2i}(\Mmb_{0,n})\right)\bigr)$ and $A_n:=\{\frac{n-4}{2},\frac{n-2}{2}\}$ if $n$ even and $A_n:=\{\frac{n-3}{2}\}$ if $n$ odd.   
\begin{itemize}
\item[{\bf (i)}] If $i \notin A_n$ then $w_{n,i}$ fulfills property $\star_{n,i}$. 
\item[{\bf (ii)}] If $i \in A_n$ then $w_{n,i}=\lambda_{n,i}$ where $\lambda_{n,i}:=\begin{cases}(4,2^{(n-4)/2})& \text{if $n$ even} \\ (4,2^{(n-5)/2},1) \quad & \text{if $n$ odd.}
\end{cases}$ 
\end{itemize}
Moreover, in the case (ii), the irreducible representation corresponding to $\lambda_{n,i}$ appears in $H^{2i}(\Mmb_{0,n})$ with multiplicity one. 
\end{theorem}

\begin{proof}
Note first that by equation \eqref{eq:maxlength} we know that if $w_{n,i}=\lambda$, then $\lambda'_1$ and $\lambda_2'$ are at most equal to ${\rm min} (i+1,n-i-2)$. 

(i)  It is clear that $w_{n,0}=(n)$ for every $n \geq 3$. Assume that $0<i<n/2-2$. We will use induction over $n$, where the base cases are covered by the statement for $w_{n,0}$. Since $H^{2i-2}(\Mmb^{n-2}_{1,2})={\rm Res}^{\Sg_{n-2}}_{\Sg_{1} \times \Sg_{n-3}}H^{2i-2}(\Mmb^{n-2}_{0,2})$, we know by induction that this cohomology group will contain a representation with character $s^x_{(1)}s^y_{\mu}$, where $\mu'_1=\mu'_2=i$. Consider the morphism $\Mmb^{n-2}_{1, 2} \rightarrow  \Mmb^{n-2}_{1, 3}$. From Theorem~\ref{thm:ind} it follows that either $s^x_{(1)}s^y_{\mu}$ will also be found in $H^{2i-2}(\Mmb^{n-2}_{1,3})$, or that there is an $m$ such that the degree $2i-2-2m$ part of $W^{n-2}_{1,m,3}$ will contain a representation with character $s^x_{(1)} s^y_{\lambda} s^z_{\nu}$ (the notation is the natural extension of the character map for $\Sg_k\times \Sg_m \times \Sg_{n-k-m(l+1)}$, compare Appendix~\ref{sec:character}), where $\nu'_1=\nu'_2=i-m$.

In the first case, we apply Theorem~\ref{thm:ind} to the morphism $\Mmb^{n}_{0, 2} \rightarrow  \Mmb^{n}_{0, 3}$ and Proposition~\ref{prop:plet} and Lemma~\ref{lem:prodw} then show that there is a contribution of the form $s^y_{(3) \cup \mu}$ in $H^{2i}(\Mmb^{n}_{0, 2})$, and thus $w_{n,i}$ fulfills property $\star_{n,i}$. 

The second case is analogous. We apply Theorem~\ref{thm:ind} to $\Mmb^{n}_{0, 2} \rightarrow  \Mmb^{n}_{0, 3}$. Proposition~\ref{prop:plet} and Lemma~\ref{lem:prodw} then tell us that $H^{2i}(\Mmb^{n}_{0,2})$ will contain a representation with character $s^y_{((2^{m+1})+\lambda') \cup \nu}$ for some partition $\lambda'$ of $m+1$. We conclude also in this case that $w_{n,i}$ fulfills $\star_{n,i}$. 
By Poincar\'e duality the statement holds for all $i \notin A_n$.

(ii) Assume that $i \in A_n$. We will use induction on $n$ to show that $w_{n,i}=\lambda_{n,i}$ and that the contribution with this character appears with multiplicity one. The base cases $w_{4,i}=(4)$, $w_{5,i}=(4,1)$, $w_{6,i}=(4,2)$ and $w_{7,i}=(4,2,1)$ are readily computed. Note that Proposition~\ref{prop:plet} and Lemma~\ref{lem:prodw} will repeatedly be used below without mention. 

If we apply Theorem~\ref{thm:ind} to the morphism $\Mmb^{n}_{0, 2} \rightarrow  \Mmb^{n}_{0, 3}$ we see that a representation with character $s^x_{\mu}s^y_{\nu}$ in $H^{2i-2m}(\Mmb^{n-2m}_{m,3})$ will contribute a representation $V$ to $H^{2i}(\Mmb^{n}_{0,2})$ with $w(\ch_n(V))=((2^m)+\mu) \cup \nu$. In order for this partition to be greater than or equal to $\lambda_{n,i}$ we must have $\nu'_1+\nu'_2 \geq n-2-2m$. Since $\nu$ is a partition of $n-3m$, this is only possible if $m \leq 2$. 

By induction, $H^{2i-4}(\Mmb^{n-4}_{2,2})={\rm Res}^{\Sg_{n-4}}_{\Sg_{2} \times \Sg_{n-6}}H^{2i-4}(\Mmb^{n-4}_{0,2})$ will contain a representation with character $s^x_{(2)}s^y_{\mu}$ and of multiplicity one, where $\mu=(2^{(n-6)/2})$ if $n$ even and $\mu=(2^{(n-7)/2},1)$ if $n$ odd. Considering the map $\Mmb^{n-4}_{2,2} \rightarrow \Mmb^{n-4}_{2,3}$ together with Theorem~\ref{thm:ind}, the argument above shows that this class will also appear in $H^{2i-4}(\Mmb^{n-4}_{2,3})$. 

Applying Theorem~\ref{thm:ind} to $\Mmb^{n}_{0, 2} \rightarrow  \Mmb^{n}_{0, 3}$ now gives us that $w_{n,i} \geq ((2^2)+(2)) \cup \mu=\lambda_{n,i}$. This will then also be the greatest contribution coming from $H^{2i-4}(\Mmb^{n-4}_{2,3})$. 

We are left to show that $w_{n,i} \leq \lambda_{n,i}$, and that the class found above is the only of its kind. We have already considered contributions from $H^{2i-2m}(\Mmb^{n-2m}_{m,3})$ for $m \geq 2$. By induction, $s^x_{(1)}s^y_{\nu}$, where $\nu=(3,2^{(n-6)/2})$ if $n$ even and $\nu=(3,2^{(n-7)/2},1)$ if $n$ odd, will appear in $H^{2i-2}(\Mmb^{n-2}_{1, 2})$ with multiplicity one. The class $s^x_{(2)}s^y_{\mu}$ in $H^{2i-4}(\Mmb^{n-4}_{2,3})$, found above, will then ensure that $s^x_{(1)}s^y_{\nu}$ will not appear in $H^{2i-2}(\Mmb^{n-2}_{1, 3})$. This deals with the case $m=1$.

To finish the proof we need to see that $w(H^{2i}(\Mmb^n_{0,3})) < \lambda_{n,i}$ for $n \geq 7$. 
We show this inductively from the bottom of the sequence $\Mmb^{n}_{0, 3} \rightarrow  \Mmb^{n}_{0, 4} \rightarrow \cdots \rightarrow \Mmb^{n}_{0, r(n,0)}$.
Lemma~\ref{lem:git2} and Lemma~\ref{lem:git3} show that $l(H^{2i}(\Mmb^n_{0,r(n,0)})) \leq 2$ and hence prove the base case.
We then show that $w(H^{2i}(\Mmb^n_{0,l})) < \lambda_{n,i}$ implies $w(H^{2i}(\Mmb^n_{0,l-1})) < \lambda_{n,i}$ for $4\leq l \leq r(n,0)$.
If $w(H^{2i}(\Mmb^n_{0,l-1}))=\lambda_{n,i}$ then, in view of Theorem~\ref{thm:ind} applied to the morphism $\Mmb^{n}_{0, l-1} \rightarrow  \Mmb^{n}_{0, l}$, 
we need to have a class with representation $s^x_{\mu}s^y_{\nu}$ in $H^{2j}(\Mmb^{n-(l-1)k}_{k,l})\otimes H^+(\mathbb{P}^{l-2})^{\otimes k}$ 
such that $w(s_{\mu}\cdot s_{\nu})=\lambda_{n,i}$ with
$\nu'_1=\lfloor(n-1)/2\rfloor-k$ for some $k >0$ and $j\in A_{n-(l-1)k}$. 
Such a class can exist only when $l=4$, $k=1$ and $n$ is even, because we have 
$l(H^{2j}(\Mmb^{n-(l-1)k}_{0,2})) \leq \lfloor(n-(l-1)k-1)/2\rfloor$  on the one hand by equation~\eqref{eq:maxlength}, 
and $l(H^{2j}(\Mmb^{n-(l-1)k}_{0,2})) \geq \nu'_1$ on the other by the sequence $\Mmb^{n-(l-1)k}_{k, 2} \rightarrow \cdots  \rightarrow \Mmb^{n-(l-1)k}_{k, l}$. 
So we are left with the case $k=1$ in $\Mmb^{n}_{0, 3} \rightarrow  \Mmb^{n}_{0, 4}$ for even $n$. But for any class $s^x_{(1)}s^y_{\nu}$ in $H^{n-6}(\Mmb^{n-3}_{1,4})$
with $\nu_1'=\frac{n-4}{2}$, we have $w(s_{(1)}\cdot s_{\nu})< \lambda_{n,i}=(4, 2^{\frac{n-4}{2}})$. This completes the proof. 
\end{proof}

\begin{corollary} For any $n \geq 3$ and $0\leq i \leq (n-3)/2$ we have
$$l\bigl(H^{2i}(\Mmb_{0,n})\bigr) = {\rm min} (i+1,n-i-2).$$
\end{corollary}

\appendix
%%%%%%%%%%%%%%%%%%%%%%%%%%%%%%%%%%%
\section{Cohomology of the blow-up} \label{sec:cohomology}
First, we recall the fact that the cohomology of the blow-up $\widetilde{M}$ of 
a smooth projective variety $M$ along a smooth subvariety $Z$ 
of codimension $l$ is given by
\begin{equation}\label{eq:general}
H^k(\widetilde{M})\cong 
H^k(M)\oplus \bigoplus_{i=1}^{l-1} \left( H^{k-2i}(Z)\otimes H^{2i}(\mathbb{P}^{l-1})\right)~,
\end{equation}
see e.g. \cite[Theorem 7.31]{Voisin}. 

Let $X$ be a smooth projective variety and $Y=Y_1 \cup \cdots \cup Y_n$ 
be the union of smooth subvarieties of $X$. Let $\widetilde{X}$ be the blow-up of $X$ 
along $Y$. We assume that any non-empty intersection of irreducible 
components of $Y$ is transversal. Then it is known that
$\widetilde{X}$ is obtained by a sequence of smooth 
blow-ups along the proper transforms of the irreducible 
components of $Y$ in any order, see e.g. \cite[Proposition 2.10]{KiemMoon}.
We further assume that $\mathrm{codim}_X Y_i=l$ for any $i$.  
The aim of this appendix is to give a formula for $H^*(\widetilde{X})$. 
For a subset $I$ of $\{1, \ldots, n\}$, we set $Y_I:=\cap_{i\in I} Y_i$,
which is a smooth subvariety of $X$ by the assumption of transversality.
\begin{proposition}\label{prop:cohomology}
Under the above assumptions, we have  
\begin{equation}\label{eq:coh-blow-up}
H^*(\widetilde{X}) \cong H^*(X) \oplus
\bigoplus_{I \subset \{1,\ldots,n\}}
\left(
H^*(Y_I) \otimes 
\left(H^+(\mathbb{P}^{l-1})
\right)^{\otimes |I|}\right)~,
\end{equation}
where $I$ runs over all subsets of $\{1,\ldots, n\}$ 
for which $Y_I \neq \emptyset$. 
\end{proposition}

\begin{proof}
Let 
$$
X_{n} \overset{\pi_n}{\longrightarrow} X_{n-1} 
\rightarrow
\cdots 
\rightarrow
X_1 
\overset{\pi_1}{\longrightarrow}
X_0
$$
be the sequence of blow-ups which we define inductively as follows. \\
(i) Let $\pi_1: X_1\to X_0:=X$ be the blow-up along $Y_1$.\\
(ii) For $i\geq 2$, we put $\pi_{i}: X_{i} \to X_{i-1}$ to be the blow-up
along the proper transform of $Y_i$. 

Note that $X_n$ is isomorphic to $\widetilde{X}$.
For $1\leq j \leq i-1$, we denote by
$Y_{i,j}$ the proper transform of $Y_i$ under
$\pi_{j}\circ \cdots \circ \pi_1: X_{j} \to X_{0}$.
In this notation, $\pi_i: X_i\to X_{i-1}$ is the blow-up along ${Y}_{i, i-1}$. 
Then it follows from \eqref{eq:general} that
\begin{equation}\label{eq:inductive}
H^*(X_i) \cong H^*(X_{i-1}) \oplus 
\left(H^*({Y}_{i,i-1}) \otimes H^+(\mathbb{P}^{l-1}) \right),
\end{equation}
as graded vector spaces.
The equation \eqref{eq:inductive} 
together with \eqref{eq:proper} given 
in Lemma \ref{lem:proper} below
implies \eqref{eq:coh-blow-up}. 
\end{proof}

\begin{lemma}\label{lem:proper}
Under the same assumptions and notation as above, we have 
\begin{equation}\label{eq:proper}
H^*({Y}_{i,i-1})\cong
H^*(Y_i) \oplus
\bigoplus_{I\subset \{1,\ldots,i-1\}} 
\left(
H^*(Y_i \cap Y_I)\otimes \left(H^+ (\mathbb{P}^{l-1})
\right)^{\otimes |I|}
\right),
\end{equation}
where $I$ runs over all the subsets of $\{1,\ldots, i-1\}$
for which $Y_i\cap Y_I \neq \emptyset$.
\end{lemma}

\begin{proof}
Note that the proper transform $Y_{i,i-1}$ of $Y_i$ is obtained by 
the sequence of blow-ups
$$
Y_{i,i-1}\overset{\pi_{i-1}}{\longrightarrow} Y_{i,i-2}\to \cdots \to Y_{i,1}
\overset{\pi_1}{\longrightarrow} Y_{i,0}:=Y_i~,
$$
and that the center of the blow-up $\pi_j: Y_{i,j}\to Y_{i,j-1}$
is $Y_{i,j-1}\cap Y_{j,j-1}$. In general, for  $I\subset \{1,\ldots, n\}$
and $j<\min \{i\in I\}$ such that $Y_{I,j}:=\cap_{i\in I}Y_{i,j}\neq \emptyset$, 
$Y_{I,j}$ is the blow-up of $Y_{I,j-1}$ along $Y_{I,j-1}\cap Y_{j,j-1}$.
Using this structure and \eqref{eq:general} recursively, we obtain 
\begin{equation}\label{eq:proper2}
H^*({Y}_{i,j})\cong
H^*(Y_i) \oplus
\bigoplus_{I\subset \{1,\ldots,j\}} 
\left(
H^*(Y_i \cap Y_I)\otimes \left(H^+(\mathbb{P}^{l-1})
\right)^{\otimes |I|}
\right),
\end{equation}
where $I$ runs over all the subsets of $\{1,\ldots, j\}$
for which $Y_I\cap Y_i \neq \emptyset$.
The desired formula \eqref{eq:proper} is the case $j=i-1$ in \eqref{eq:proper2}. 
\end{proof}

%%%%%%%%%%%%%%%%%%%%%%%%%%%%%%%%%%%%%%%%%%%%%%%%%%%%%%%%%%%
\section{Characters of representations of symmetric groups} \label{sec:character}
Let $\Sg_n$ be the symmetric group on $n$ letters.
Let $\Lambda:=\underset{\longleftarrow}{\lim}~\mathbb{Z}[x_1, \ldots, x_n]^{\Sg_n}$ 
be the ring of symmetric functions. 
It is well known that $\Lambda \otimes \mathbb{Q}=\mathbb{Q}[p_1, p_2, \ldots]$ where $p_n$ are the power sums. We denote by $\Par(n)$ the set of partitions of $n$, and for $\lambda=(\lambda_1,\ldots,\lambda_{l(\lambda)}) \in \Par(n)$ we set $p_{\lambda}:=\prod_{i=1}^{l(\lambda)} p_{\lambda_i}$. 
We also set $\Lambda^{x, y}:=\Lambda^x\otimes \Lambda^y$, where 
$\Lambda^x$ and $\Lambda^y$ are the ring of 
symmetric functions in $x=(x_1, x_2, \ldots)$ and $y=(y_1, y_2, \ldots)$
respectively.

For a representation $V$ of  $\Sg_n$, we define 
\begin{equation*}
\ch_n(V)
:=\frac{1}{n!} \sum_{w \in \Sg_n}
{\rm Tr}_V(w) p_{\rho(w)} \in \Lambda \, ,
\end{equation*}
where $\rho(w)\in \Par(n)$ is the cycle type of $w\in \Sg_n$. Similarly we define, for an $\Sg_k \times \Sg_{n-k}$ representation $V$, 
\begin{equation*}
\ch_{k, n-k}(V)
:=\frac{1}{k! }\frac{1}{(n-k)!} \sum_{(u,v) \in \Sg_k \times \Sg_{n-k}}
{\rm Tr}_V\bigl((u,v)\bigr) p_{\rho(u)}^x \, p_{\rho(v)}^y \in \Lambda^{x, y} \, ,
\end{equation*}
where $p_n^x$ and $p_n^y$ 
are the power sums in the variable $x$ and $y$ respectively.

If $V$ and $W$ are representations of $\Sg_n$ we put
$$\ch_{n}(V) * \ch_{n}(W):=\ch_{n}(V \otimes W).$$
For any $\lambda \in \Par(k)$, if we put $m_j(\lambda):=\#\{ i \mid \lambda_i =j\}$ and 
$$\frac{\partial}{\partial p^x_{\lambda}}:=\Bigl(\prod_{i=1}^{k}\frac{1}{m_i(\lambda)!} \Bigr) \frac{\partial}{\partial p^x_{\lambda_1}} \frac{\partial}{\partial p^x_{\lambda_2}} \cdots \frac{\partial}{\partial p^x_{\lambda_{l(\lambda)}}},$$
then 
$$\ch_{k,n-k}\bigl({\rm Res}^{\Sg_{n}}_{\Sg_{k} \times \Sg_{n-k}}(V)\bigr)=\sum_{\lambda \in \Par(n-k)} 
\Bigl( \frac{\partial}{\partial p^x_{\lambda}} \ch_{n}(V) \Bigr) \, p^y_{\lambda}.$$
If $V_i$ are representations of $\Sg_{n_i}$ for $1 \leq i \leq k$ then  
$$\ch_{\sum_{i=1}^k n_i} \Bigl({\rm Ind}_{\Sg_{n_1} \times \ldots \times \Sg_{n_k}}^{\Sg_{\sum_{i=1}^k n_i}} (V_1 \boxtimes V_2 \boxtimes \ldots \boxtimes V_k) \Bigr) = \prod_{i=1}^k \ch_{n_i}(V_i).$$

If $\circ$ denotes plethysm between symmetric functions, 
and $\sim$ denotes the wreath product, that is,  $\Sg_{n_1} \sim\,  \Sg_{n_2}:=\Sg_{n_1} \ltimes (\Sg_{n_2})^{n_1}$ where $\Sg_{n_1}$ acts on $(\Sg_{n_2})^{n_1}$ by permutation, then 
$$\ch_{n_1 n_2}\Bigl({\rm Ind}_{\Sg_{n_1} \sim \, \Sg_{n_2}}^{\Sg_{n_1 n_2}} (V_1 \boxtimes \underbrace{ V_2 \boxtimes \ldots \boxtimes V_2 }_{n_1} \,) \Bigr) = \ch_{n_1}(V_1) \circ \ch_{n_2}(V_2),
$$
see~\cite[Appendix A, p.~158]{Mac}. 

Recall finally that irreducible representations of $\Sg_n$ are indexed by $\Par(n)$. For $\lambda \in \Par(n)$, let $V_{\lambda}$ be the irreducible representation corresponding to 
$\lambda$ and define the Schur function
$$
s_{\lambda}:=\ch_n(V_{\lambda})\in \Lambda \, .
$$
It is well-known that $\{s_{\lambda}\}$,  
where $\lambda$ runs over all the partitions, 
is a $\mathbb{Z}$-basis of $\Lambda$.

%%%%%%%%%%%%%%%%%%%%%%%%%%%%%%%%%%%%%%%%%%%%
\section{An ordering of symmetric functions} \label{sec:order} 
For any partition $\lambda=(\lambda_1,\lambda_2,\ldots,\lambda_{l(\lambda)})$ we denote its dual partition by $\lambda'=(\lambda'_1,\lambda'_2,\ldots)$, where $\lambda'_i:= |\{j:\lambda_j \geq i \}|$. If $\mu=(\mu_1,\mu_2,\ldots,\mu_{l(\mu)})$ is another partition we define the partitions $\lambda+\mu$ as $(\lambda_1+\mu_1,\lambda_2+\mu_2,\ldots)$ and the partition $\lambda \cup \mu$ as the reordering of $(\lambda_1,\ldots,\lambda_{l(\lambda)},\mu_1,\mu_2,\ldots,\mu_{l(\mu)})$. Note that $\lambda \cup \mu=(\lambda'+\mu')'$.

We introduce an ordering. For any partitions $\lambda$ and $\mu$ we say that $\lambda > \mu$ if there is a $k$ such that $\lambda'_i=\mu'_i$ for $1 \leq i \leq k-1$ and $\lambda'_k > \mu'_k$.

\begin{definition} For any symmetric function $f=\sum_{\lambda} a_{\lambda} s_{\lambda}$ we let $w(f)$ be the maximal partition $\lambda$ (w.r.t. $>$) such that $a_\lambda \neq 0$. 
\end{definition}

For any partition $\lambda$, we put $h_{\lambda}:=\prod_{i=1}^{l(\lambda)} s_{(\lambda_i)}$ and $e_{\lambda}:=\prod_{i=1}^{l(\lambda)} s_{(1^{\lambda_i})}$. The following is well known. 

\begin{lemma} \label{lem:rt}
There are integers $a_{\lambda,\mu}$ and $b_{\lambda,\mu}$ such that 
$$s_{\lambda}=h_{\lambda}+\sum_{\mu < \lambda} a_{\lambda,\mu}h_{\mu}=e_{\lambda'}+\sum_{\mu < \lambda'} b_{\lambda,\mu}e_{\mu}.$$
\end{lemma}

\begin{lemma} \label{lem:prodw} For any symmetric functions $f$ and $g$ 
we have, $w(fg)=w(f) \cup w(g).$
\end{lemma}
\begin{proof} 
Since $h_{\mu} h_{\nu}=h_{\mu \cup \nu}$, it follows from Lemma~\ref{lem:rt} that there are integers 
$c_{\lambda}$ such that $fg=c_{w(f) \cup w(g)}h_{w(f) \cup w(g)}+\sum_{\lambda < w(f) \cup w(g)} c_{\lambda} h_{\lambda}$
with $c_{w(f) \cup w(g)}\neq 0$.
\end{proof}

Due to the lack of a suitable reference we will show here how $w$ behaves with respect to plethysm between symmetric functions. 

\begin{lemma}\cite[p.~158]{Mac}\label{lem:plet:prod} 
For any symmetric functions $f$, $g$ and $h$ we have, 
\begin{equation*} 
(f \circ h) (g \circ h)=(f \, g) \circ h.
\end{equation*}
\end{lemma}

We denote by $(\cdotp | \cdotp )$ the standard inner product on $\Lambda$ for which Schur functions are orthonormal. 

\begin{proposition}\cite[Theorem I, IA]{Newell} Say that $s_{(k)} \circ s_{(m)}=\sum_{\lambda} a_{\lambda} s_{\lambda}$, then, for any $0 \leq i \leq k$, 
  \begin{equation}
    \label{eq:newell1}
    (s_{(1^i)} \circ s_{(m-1)}) \, (s_{(k-i)} \circ s_{(m)}) = \sum_{\lambda} \sum_{\nu} a_{\lambda} \bigl(s_{(1^i)}s_{\nu} | s_{\lambda}\bigr) s_{\nu}.
  \end{equation}
Similarly, say that $s_{(1^k)} \circ s_{(m)}=\sum_{\lambda} b_{\lambda} s_{\lambda}$, then, for any $0 \leq i \leq k$, 
\begin{equation}
  \label{eq:newell2}
  (s_{(i)} \circ s_{(m-1)}) \, (s_{(1^{k-i})} \circ s_{(m)}) = \sum_{\lambda} \sum_{\nu} b_{\lambda} \bigl(s_{(1^i)}s_{\nu} | s_{\lambda}\bigr) s_{\nu}.
\end{equation} 
\end{proposition}

\begin{proposition} \label{prop:plet}
For any $m \geq 1$ 
and partition $\mu$ of $k$ we have
\begin{equation} \label{eq:plet}
w(s_{\mu} \circ s_{(m)})=\begin{cases} {((m-1)}^k) + \mu \quad \text{if $m$ odd} \\ ({(m-1)}^k) +  \mu'  \quad \text{if $m$ even.}\end{cases}
\end{equation}
\end{proposition}
\begin{proof} We first recall that $l(s_{\mu} \circ s_{(m)}) \leq |\mu|$, see \cite[Example 9, p. 140]{Mac}. Let us then continue by proving Equation~\eqref{eq:plet} for $\mu=(k)$ and $\mu=(1^k)$ by induction on $m$. The equation clearly holds for $m=1$. If $l(s_{\nu})=k$ and $l(s_{\lambda}) \leq k$ then $(s_{(1^k)}s_{\nu}|s_{\lambda}) \neq 0$ implies that $\lambda=\nu + (1^k)$. Moreover, if $l(s_{\eta})<k$, $l(s_{\lambda}) \leq k$ and $(s_{(1^k)}s_{\eta}|s_{\lambda}) \neq 0$ then $\lambda < \nu + (1^k)$ for any $\nu$ such that $l(\nu)=k$. Therefore, taking $i=k$ in Equation~\eqref{eq:newell1} and applying the fact that $l(s_{(k)} \circ s_{(m)}) \leq k$ we find by looking at the term with $\lambda=w(s_{(k)}\circ s_{(m)})$ on the right hand side of formula \eqref{eq:newell1}
that $w(s_{(k)} \circ s_{(m)})=w(s_{(1^k)} \circ s_{(m-1)}) + (1^k)$. Similarly, using Equation~\eqref{eq:newell2} we find by induction on $m$ that $w(s_{(1^k)} \circ s_{(m)})=w(s_{(k)} \circ s_{(m-1)}) + (1^k)$. This completes the induction. 

Finally, let us prove the statement for any $\mu$ and $m$. From Lemma~\ref{lem:plet:prod} and from the formula for $w(s_{(k)} \circ s_{(m)})$ it follows, for $m$ odd, that 
$$w(h_{\mu} \circ s_{(m)})= \bigcup_i w(s_{(\mu_i)} \circ s_{(m)})={((m-1)}^k) + \mu $$
and similarly from the formula for $w(s_{(1^k)} \circ s_{(m)})$, for $m$ even, that
$$w(e_{\mu} \circ s_{(m)})= \bigcup_i w(s_{(1^{\mu_i})} \circ s_{(m)})={((m-1)}^k) + \mu. $$
Using Lemma~\ref{lem:rt} it then follows, for $m$ odd, that 
$$w(s_{\mu} \circ s_{(m)})=w \bigl((h_{\mu}+\sum_{\nu < \mu} a_{\nu}h_{\nu}) \circ s_{(m)} \bigr)=w(h_{\mu} \circ s_{(m)})$$
and, for $m$ even, that 
$$w(s_{\mu} \circ s_{(m)})=w\bigl((e_{\mu'}+\sum_{\nu < \mu'} b_{\nu} e_{\nu}) \circ s_{(m)} \bigr)=w(e_{\mu'} \circ s_{(m)}).$$
\end{proof}

%%%%%%%%%%%%%%%%%%%%%%%%%%%

\end{document}